\documentclass{article}

\usepackage{amssymb,amsmath,xcolor,graphicx,xspace,colortbl,textcomp, rotating} 
\graphicspath{{BooleanPME_graphics/}{BooleanPME_tcache/}{BooleanPME_gcache/}}
\DeclareGraphicsExtensions{.pdf,.eps,.ps,.png,.jpg,.jpeg}\newtheorem {theorem}{Theorem}

\newtheorem {corollary}[theorem]{Corollary}

\newtheorem {definition}[theorem]{Definition}

\newtheorem {exercise}[theorem]{Exercise}
\newtheorem {lemma}[theorem]{Lemma}

\newtheorem {proposition}[theorem]{Proposition}
\newtheorem {remark}[theorem]{Remark}

\newenvironment {proof}[1][Proof]{\noindent \textbf {#1.} }{\ \rule {0.5em}{0.5em}}
\begin{document}
\begin{center}{\Large Primary Decomposition in Boolean Rings}\end{center}\par
\begin{center}
David C. Vella, \end{center}\par
\begin{center}Skidmore College\end{center}\par

{\textbf{1. Introduction}
}
Let $R$ be a commutative ring with identity. The Lasker-Noether theorem on primary decomposition states that if $R$ is Noetherian, every ideal of $R$ can be expressed as a finite intersection of primary ideals (defined below) in such a way that the decomposition is irredundant, and there is a certain amount of uniqueness to the decomposition (such decompositions are called \emph{reduced} or \emph{normal}.) For details of the proof of this theorem, see [ZS] or [N] for ideals, and [H] or [AM] for a more modern treatment that applies to $R$-modules as well as ideals.  This beautiful theorem is at once a generalization of the fundamental theorem of arithmetic for the integers, which states that every integer is a product of prime powers in an essentially unique way, and at the same time it is an algebraic version of the geometric fact that algebraic varieties are unions of a finite number of irreducible varieties. \bigskip

The beginning commutative algebra student who studies this theorem may want to see an example of an ideal in a ring which does \textit{not} have a primary decomposition. Such a counterexample must of course come from a ring which is not Noetherian, but in what ring should one look?  \bigskip 

There is a canonical example of such an ideal, which is alluded to in exercise 6 of Chapter 4 in [AM].  Let $X$ be an infinite compact Hausdorff space, and $R$ be the ring of continuous real-valued functions $R =\mathcal{C}[X ,\mathbb{R}]$ on $X .$  The essence of the example is that (i), for each $x \in X ,$ the set $m_{x} =\{f \in R$ \textbar{} $f(x) =0\}$ is a maximal ideal of $R ,$ (ii) each $m_{x}$ contains a minimal prime ideal $p_{x}$, (iii) The zero ideal can be written as the intersection $(0) = \cap _{x \in X}p_{x} ,$ and (iv) if $x \neq y ,$ then $p_{x} \neq p_{y} ,$ which means the decomposition in (iii) is infinite and can't be refined to a finite, reduced decomposition.  Most of the details are elementary, but to prove point (iv), one must rely on Urysohn's lemma from topology (see Exercise 26 of Chapter 1 of [AM]). \bigskip

A beautiful counterexample to be sure, yet there is something unsatisfying about it. Why should this seemingly purely algebraic question seem to rest on a topological result?  Could we not find a counterexample which is more algebraic (and perhaps more elementary)? \bigskip

The main result of this note is the following: \bigskip

\begin{theorem}
Let $X$ be a set, and let $R =\mathcal{P}(X)$ be the power set of $X ,$ which is a Boolean ring with addition in $R$ being the symmetric difference of sets, and multiplication being the intersection of sets.  Then the zero ideal $(0)$ has a reduced primary decomposition in $R$ if and only if $X$ is finite. \bigskip 
\end{theorem}

Thus, if we take an infinite set, such as $\mathbb{N}$ for example, then the zero ideal in $R =\mathcal{P}(\mathbb{N})$ does not have a primary decomposition.  We will phrase part of our argument in terms of arbitrary Boolean rings, and indeed much of what we prove in the case of $R =\mathcal{P}(X)$ will generalize to all Boolean rings.  But we would like to work with this specific Boolean ring so that at the end, we can show the connection between our counterexample and the standard counterexample from [AM] mentioned above. \bigskip 

{\textbf{2.   Maximal, Prime, and Primary Ideals in a Commutative Ring.}}
We review some basic notions of commutative algebra in order to give a more precise statement of the Lasker-Noether theorem.  Readers familiar with these concepts may skip to the next section. 

Commutative algebra is the study of ideals in commutative rings.  We assume familiarity with the notions of commutative rings and ideals, and basic facts about quotient rings $R/I$ such as the first isomorphism theorem ([G], theorem 15.3.)  All commutative rings $R$ in this paper will contain a multiplicative identity element $1.$ There are a few special types of ideals to single out, most of which are encountered in a typical first course of abstract algebra.  For those readers needing a review, one standard undergraduate reference is [G]. Any notions not found there can be found in [H] or [AM]. \bigskip

\begin{definition}
A proper Ideal $I$ in a commutative ring $R$ is a \emph{maximal ideal} if whenever $J$ is an ideal such that $I \subseteq J \subseteq R ,$ then either $J =I$ or $J =R .$
\end{definition}

We remind the reader of the following standard result about commutative rings: \bigskip

\begin{lemma}
In a commutative ring $R$ every proper ideal is contained in at least one maximal ideal.
\end{lemma}

\begin{proof}
This is a standard application of Zorn's lemma.  See (1.3 ) and (1.4) of [AM]. \bigskip 
\end{proof}

We also remind the reader of the fact that $I$ is maximal if and only if $R/I$ is a field.  (See, for example, Theorem 14.4 of [G].) \bigskip

\begin{definition}
A proper ideal in $R$ is a\emph{ prime ideal }if it has the property that whenever $ab \in I$ and $a \notin I ,$ then $b \in I .$
\end{definition}

\begin{definition}
A proper ideal $I$ in $R$ is a \emph{primary ideal} if it has the property that whenever $ab \in I$ and $a \notin I ,$ then $b^{n} \in I$ for some $n >0.$
\end{definition}

\begin{definition}
Given an ideal $I ,$ the \emph{radical} of $I ,$ denoted $\sqrt{I} ,$ is $\{x \in R$ \textbar{} $x^{n} \in I$ for some $n >0\} .$ It is also an ideal of $R$.  An ideal with the property that $\sqrt{I} =I$ is said to be a \emph{radical ideal.}
\end{definition}

Elementary facts about these types of ideals include the following:

\begin{itemize}
\item $I \subseteq \sqrt{I}$

\item
Every maximal ideal is a prime ideal

\item
Every prime ideal is a primary ideal

\item
Every prime ideal is a radical ideal

\item
If $I$ is primary, then $\sqrt{I}$ is a prime ideal\end{itemize}

The proofs of these claims are easy exercises, or can be found in the usual references such as [G], [H], [AM], or [N]. \bigskip 

Let $a_{1} ,a_{2} , . . . ,a_{n}$ be elements of $R .$  We use the usual notation $(a_{1 ,}a_{2} , . . . ,a_{n})$ to stand for the ideal generated by these elements, which is the set of linear combinations $\{\sum r_{i}a_{i}$ \textbar{} $r_{i} \in R\} .$  The set $\{a_{1} ,a_{2} , . . . ,a_{n}\}$ are called the generators of $I =(a_{1} , . . . ,a_{n}) .$  If $I$ has a finite set of generators, we say it is \emph{finitely generated}, and if $I =(a)$ has a single generator, we say it is a \emph{principal ideal}.  \bigskip

The ring $\mathbb{Z}$ of integers is something of a special case.  In $\mathbb{Z}$, every ideal is principal (see [G], theorem 18.4), and an ideal $I =(m)$ is prime if and only if it is generated by a prime number $m =p$, while it is primary if and only if it is generated by a power of a prime number $m =p^{r};$ for some $r >0.$  Also, in $\mathbb{Z} ,$ every prime ideal is maximal, which is certainly not true in general.  But there are some other rings in which every prime ideal is maximal, as we shall see. \bigskip

We define a commutative ring $R$ to be \emph{Noetherian} if every ideal of $R$ is finitely generated.  This definition is equivalent to the usual definition in terms of ascending chains of ideals (see [H], Chapter VIII, theorem 1.9.) \bigskip

The following example will illustrate the above notions.  In $\mathbb{Z} ,$ the fundamental theorem of arithmetic says that every integer factors uniquely as a product of prime elements.  For example, the unique decomposition of $360$ is:
\begin{equation}360 =2^{3} \cdot 3^{2} \cdot 5
\end{equation}

Now consider the ideal $(360) .$  We leave it as an (easy!) exercise to verify that:\begin{equation}(360) =(2^{3}) \cap (3^{2}) \cap (5)
\end{equation}

This is an example of primary decomposition.  The ideals $(2^{3}) =(8) ,$ $(3^{2}) =(9) ,$ and $(5)$ are primary ideals.  Note that their radicals are prime: $\sqrt{(8)} =(2) ,$ $\sqrt{(9)} =(3) ,$ and $\sqrt{(5)} =(5)$ since $(5)$ is already prime and hence a radical ideal.  Thus the unique prime factorization of the element $360$ leads to a unique primary decomposition of the ideal $(360) ,$ and this example makes it clear that the same holds for any ideal in $\mathbb{Z} .$

Of course, unique factorization of elements is a rare thing in commutative rings.  The beauty of the Lasker-Noether theorem is that even if unique factorization fails for elements of the ring, then (as long as the ring is Noetherian), we still have primary decomposition of ideals.  Emanuel Lasker proved this first for polynomial rings in 1905, and Emmy Noether generalized it to all Noetherian rings in 1921.  In the general statement of the theorem, the uniqueness of the decomposition is more complicated than in the case for $\mathbb{Z} .$  In general, if \begin{equation}I =Q_{1} \cap Q_{2} \cap  . . . \cap Q_{n}
\end{equation}

is a decomposition of $I$ into primary ideals $Q_{i} ,$ let $P_{i} =\sqrt{Q_{i}}$ be the radical of $Q_{i}$.  As noted above, the $P_{i}$ are prime ideals.  However, unlike the case when $R =\mathbb{Z} ,$ these $P_{i}$ are not necessarily maximal ideals, and so it is possible that there may be containment relations between them.  It is also possible that several different $Q_{i}$ have the same radical - however, in a \textit{reduced} primary decomposition, the $n$ is minimal (in particular, no $Q_{i}$ contains the intersection of the other primary ideals in the decomposition), and the $P_{i}$ are distinct.  Then the statement of the theorem is:

\begin{theorem}
(Lasker-Noether)  If $R$ is a Noetherian ring, then every ideal has a reduced primary decomposition $I = \cap _{i =1}^{n}Q_{i}$. $Let P_{i} =\sqrt{Q_{i}}$ be the radicals of the factors.  In general, the primary decomposition is not unique.  However, the $\{P_{i}\}$ are uniquely determined by $I ,$ and furthermore, the $Q_{i}$ which correspond to the minimal elements of the set $\{P_{i}\}$ are also uniquely determined by $I .$
\end{theorem}

In particular, if there are no containments among the primes $\{P_{i}\} ,$ then every $P_{i}$ is minimal in the set $\{P_{i}\} .$  In this case, all the $Q_{i}$ are uniquely determined by $I ,$ so that the reduced primary decomposition of $I$ is unique in this case.  Notice that if $R$ is a ring in which every prime ideal is maximal (such as $\mathbb{Z}$), then there cannot be any containments between distinct prime ideals, so in such a ring, a reduced primary decomposition of any ideal is always unique.

Readers interested in seeing examples of ideals with several distinct reduced primary decompositions may consult ([H], Chapter VIII, Section 2, exercise 18.) \bigskip 

\textbf{3. Boolean Rings and Power Set Rings}.  A \emph{Boolean Ring} is a ring $R$ with identity $1$ such that every element $x \in R$ is \emph{idempotent}; that is, $x^{2} =x$ for all $x \in R .$  We summarize standard facts about Boolean rings in the following:

\begin{proposition}
  Let $R$ be a Boolean ring.  Then:

1) $R$ has characteristic $2;$ that is, $2x =0$ for all $x \in R .$

2) $R$ is commutative.

3) Every primary ideal in $R$ is in fact a prime ideal.

4) Every prime ideal in $R$ is in fact maximal.

5) Every finitely generated ideal in $R$ is in fact principal.
\end{proposition}

\begin{proof}
  These results are well-known, but we include the proofs to make this note more self-contained.  

For (1), use the fact that $(x +x)^{2} =x +x ,$ whence expanding the left side gives $4x^{2} =2x ,$ or $4x =2x$ since $x^{2} =x .$  Therefore, $2x =0$ for any $x .$  

For (2), expand $(x +y)^{2} =x +y$ to obtain $x^{2} +xy +yx +y^{2} =x +y ,$ or $xy +yx =0 ,$ since $x^{2} =x$ and $y^{2} =y .$  Thus, $xy = -yx ,$ but by part (1), $ -yx =yx ,$ so $xy =yx .$

For (3) recall that an ideal $I$ is \emph{primary} if $xy \in I$ and $x \notin I$ imply $y^{n} \in I$ for some $n >0.$  But in a Boolean ring, an easy induction gives that $y^{n} =y$ for all $n >0.$   Thus, in a Boolean ring, if $I$ is primary, then $xy \in I$  and $x \notin I$ imply $y \in I ,$ which is precisely the definition of a \emph{prime} ideal.

For (4), Let $P$ be a prime ideal, and let $Q$ contain $P$ properly.  Then there is an element $x \in Q$ such that $x \notin P .$  But $x =x^{2} ,$ whence $x -x^{2} =0;$ that is,  we have $x(1 -x) =0 \in P$ and $x \notin P .$  Since $P$ is prime, this yields $1 -x \in P \subseteq Q ,$ and since $x \in Q ,$ this gives, $x +(1 -x) =1 \in Q ,$ so $Q =R .$  Thus $P$ is maximal.

Finally, for (5) Let $I =(x_{1} ,x_{2} , . . . ,x_{n})$ be a finitely generated ideal.  By induction we can reduce to the case $n =2.$  But $(x ,y) =(x +y +xy) .$  Indeed the containment $ \supseteq $ is obvious.  For the reverse, using the idempotence of elements and the fact that $R$ has characteristic $2 ,$ we have:

\begin{align}ax +by =ax^{2} +by^{2} =ax^{2} +axy +axy +by^{2} +bxy +bxy \\
 =ax^{2} +axy +ax^{2}y +by^{2} +bxy +bxy^{2} \\
 =ax(x +y +xy) +by(x +y +xy) \\
 =(ax +by)(x +y +xy)\end{align}

showing $(x ,y) \subseteq (x +y +xy) .$  This completes the proof.
\end{proof}

\bigskip The following are examples of Boolean rings:\bigskip

1) The field $\mathbb{Z}_{2}$ of integers mod $2.$  (It is the only Boolean ring which is a field.) \medskip

2) The direct product $\mathbb{Z}_{2}$$ \times \mathbb{Z}_{2} \times  . . . \times \mathbb{Z}_{2}$ of any number of copies of $\mathbb{Z}_{2} .$ \medskip

3) Let $X$ be any set, then as mentioned in the introduction, $R =\mathcal{P}(X)$, the power set of $X$ is a Boolean ring.  Here $A +B =(A \cup B) -(A \cap B)$ is the symmetric difference of $A$ and $B ,$ and $AB =A \cap B$ is the intersection.  The empty set $ \varnothing $ is the $0$ in this ring and the set $X$ is the multiplicative identity $1.$  It is easy to see that $1$ is the only unit in this ring, and $R$ is never an integral domain unless $X$ is a singleton set, in which case $R$ is isomorphic to $\mathbb{Z}_{2}$. \medskip

4) Let $X$ be an infinite set, and consider the subset $Fin(X) \subseteq \mathcal{P}(X)$ consisting of all finite subsets of $X .$  Then this is a Boolean ring - in fact, we leave the easy proof that it is a (proper) ideal of $\mathcal{P}(X)$ to the reader, and it is generated by all the singleton sets in $\mathcal{P}(X) .$  Furthermore, none of the singletons can be omitted from the generating set, so this is an example of an ideal in $\mathcal{P}(X)$ which is not finitely generated.  In particular, when $X$ is infinite, $\mathcal{P}(X)$ is not Noetherian.  On the other hand, when $X$ is finite, so is $\mathcal{P}(X) ,$ so it is rather trivially a Noetherian ring in this case.

\begin{exercise}
If $A \subseteq X ,$ then the principal ideal $(A)$ generated by $A$ in $\mathcal{P}(X)$ is equal to $\mathcal{P}(A)$ (not just isomorphic to $\mathcal{P}(A)$, but equal to it!)
\end{exercise}

Note that if $X$ is finite, then as noted above $R =\mathcal{P}(X)$ is Noetherian, and so by part (5) of Proposition 8 above, $R$ is a principal ideal ring (but not a domain) in this case.  In the case $X$ is finite, with $n$ elements, after doing exercise 9, the reader would probably speculate that the maximal ideals in $R$ must all have the form $\mathcal{P}(A)$, where $A$ is a subset with $n -1$ elements.  This is correct: \bigskip

\begin{proposition}
Let $R =\mathcal{P}(X)$ be a power set ring. If $x \in X ,$ there is a maximal principal ideal of the form $\mathcal{P}(X -\{x\})$.  Denote this ideal by $m_{x} .$  If $X$ is finite, then every maximal ideal of $R$ is one of the $m_{x} .$  If $X$ is infinite, there exist maximal ideals of $R$ which are not of the form $m_{x} .$  Also, $x \neq y$ implies that $m_{x} \neq m_{y} .$
\end{proposition}

\begin{proof}
Recall that for a set $A ,$ the characteristic function $\chi _{A}(x) =1$ if $x \in A ,$ and $\chi _{A}(x) =0$ otherwise.  Usually, we think of this as a real-valued function, but since the only two values of the function are $0$ and $1 ,$ we may as well consider it to be $\mathbb{Z}_{2}$ -valued.  Now by holding $x$ fixed and letting $A$ vary, we can regard this as a function on $\mathcal{P}(X)$ rather than on $X .$  That is, given a fixed $x \in X ,$ consider the function:
\begin{align}f_{x} :\mathcal{P}(X) \rightarrow \mathbb{Z}_{2} \\
f_{x}(A) =\chi _{A}(x) =\{\begin{array}{c}1\text{  if}\thinspace \text{}\text{}x \in A \\
0\text{  if}\text{}\thinspace x \notin A\end{array}\end{align}

We claim this is a surjective ring homomorphism.  Indeed, $f_{x}(A +B) =f_{x}((A \cup B) -(A \cap B)) =1$ precisely when $x$ belongs to exactly one of $A$ or $B$, but this is also the case for $f_{x}(A) +f_{x}(B) ,$ so $f_{x}$ preserves sums.  Similarly, $f_{x}(AB) =f_{x}(A \cap B) =1$ precisely when $x$ belongs to both $A$ and $B ,$ but this is also the case for $f_{x}(A)f_{x}(B) ,$ so $f_{x}$ preserves products as well.  The map is clearly surjective if $X$ is nonempty, since $f_{x}( \varnothing ) =0$ and $f_{x}(\{x\}) =1.$  It follows from the first isomorphism theorem that $\mathbb{Z}_{2} \approxeq R/\ker (f_{x}) ,$ and since $\mathbb{Z}_{2}$ is a field we obtain that $\ker (f_{x})$ is a maximal ideal.

Looking more carefully, we see that \begin{equation}\ker (f_{x}) =\{A \subseteq X\vert x \notin A\} =\mathcal{P}(X -\{x\}) =m_{x} ,
\end{equation}

the principal ideal of $R$ generated by $X -\{x\} .$  So we obtain one maximal (principal) ideal for each element $x \in X .$  Furthermore, if $x \neq y ,$ then $\{x\} \notin m_{x} ,$ but $\{x\} \in m_{y} ,$ so it follows that $m_{x} \neq m_{y}$ if $x \neq y .$

Now suppose $m$ is a maximal ideal in $R ,$ and $m$ is not contained in any $m_{x} .$  That means for each $x \in X ,$ $m$ contains a subset $A_{x}$ of $X$ which contains $x$.  Thus, $m$ contains the ideal generated by these $A_{x} :$ $(A_{x}\vert x \in X) \subseteq m .$  Therefore, any \textit{finite} sum of the form $\sum B_{x}A_{x}$ belongs to $m ,$ where $B_{x}$ is an arbitrary subset of $X .$  Taking $B_{x} =\{x\}$ and remembering that multiplication means intersection (and that $x \in A_{x}) ,$ this says that any finite sum of the form $\sum _{x \in X}\{x\}$ belongs to $m .$  In other words, every finite subset of $X$ belongs to $m$ (the sum of distinct singleton sets is their union because they are disjoint.)  We have shown $Fin(X) \subseteq m .$  Now if $X$ is finite, then $Fin(X) =\mathcal{P}(X)$ so in particular $X \in m$ so $m =R ,$ contradicting the fact that it is a maximal ideal.  Thus if $X$ is finite, no maximal ideals exist which are not contained in some $m_{x} .$  On the other hand, if $m \subseteq m_{x} ,$ then they are equal by maximality of $m$, whence for finite sets $X ,$ the ideals $m_{x}$ exhaust the maximal ideals of $R =\mathcal{P}(X) .$

However, if $X$ is infinite, then $Fin(X)$ is a proper ideal of $R ,$ and by lemma 3, it is contained in some maximal ideal $m$ of $R ,$ which clearly is not contained in any $m_{x} .$  Thus maximal ideals exist which are not of the form $m_{x}$ (in particular, they are not principal, hence not finitely generated.)  This completes the proof. \bigskip 
\end{proof}

In short, for power set rings $R =\mathcal{P}(X)$, we know that the maximal spectrum is the same as the prime spectrum, $Max(R) =Spec(R)$ (true for any Boolean ring as we showed above.)  In the case when $X$ is finite, the map $m_{x} \mapsto x$ gives a bijection of $Spec(R)$ with $X$ itself.  When $X$ is infinite, this map is only a bijection with the \textit{principal} prime ideals in $R ,$ and the full spectrum contains other (infinitely-generated) maximal ideals, all of which must contain $Fin(X)$. \bigskip  \bigskip

We are almost ready to prove the main result.  We need just one more general fact about prime ideals: \bigskip

\begin{lemma}
(Proposition 1.11, part (ii) of [AM])  Let $I_{1} ,I_{2} , . . . ,I_{n}$ be ideals in a commutative ring and suppose $P$ is a prime ideal containing $ \cap _{j =1}^{n}I_{j} .$  Then $P \supseteq I_{k}$ for some $k .$
\end{lemma}

\begin{proof}
  Suppose, by way of contradiction, that no $I_{k} \subseteq P$.  Then for each $k ,$ there is an element $x_{k} \in I_{k}$ such that $x_{k} \notin P .$  Since $P$ is prime, it follows that the product $x_{1}x_{2} \cdot  . . . \cdot x_{n} \notin P .$  But this product belongs to $ \cap _{j =1}^{n}I_{j} \subseteq P ,$ a contradiction. \bigskip \bigskip 
\end{proof}

{\textbf{4. Proof of the Main Result}.}
We now turn to the proof of Theorem 1.  Let $X$ be a set and let $R =\mathcal{P}(X)$ be its power set ring.  Observe that the following equality holds for the zero ideal:\begin{equation}\text{ }(0) =\bigcap _{x \in X}m_{x}
\end{equation}

Indeed, a subset $A$ of $X$ which belongs to $m_{x}$    must not contain $x ,$ so if $A$ is a member of every \thinspace $m_{x} ,$ it must not contain any $x \in X ,$ whence $A = \varnothing $.  Thus, the only set in the right side is the empty set, which is the ideal $(0) .$  Now if $X$ is finite, then this is a reduced primary decomposition of $(0) ,$ since it is a finite intersection of maximal (hence primary) ideals, and $x \neq y$ forces $m_{x} \neq m_{y} ,$ so the decomposition is irredundant. Thus, $(0)$ has a primary decomposition in this case.  (Of course, since $R$ is Noetherian if $X$ is finite, the Lasker-Noether theorem directly implies $(0)$ has a primary decomposition, but here we have exhibited a specific one.)

Now if $X$ is infinite, this particular decomposition is not a reduced primary decomposition since it is infinite.  Moreover, the presence of maximal (hence primary) ideals which are not of the form $m_{x}$ appears to complicate the situation.  But suppose there exists a reduced primary decomposition:\begin{equation}(0) =\bigcap \limits _{j =1}^{n}Q_{j}
\end{equation}

For a fixed $x ,$ observe that $(0)$ $ \subseteq m_{x};$ that is, \begin{equation}m_{x} \supseteq \bigcap \limits _{j =1}^{n}Q_{j}
\end{equation}

Since $m_{x}$ is prime, it follows immediately from lemma 11 that $m_{x} \supseteq Q_{k}$ for some $k \in \{1 ,2 , . . . ,n\} .$  However, since $Q_{k}$ is primary, it is maximal by proposition 8, whence $Q_{k} =m_{x} .$ Since $k$ is determined by $x ,$ this defines a map : $\omega  :X \rightarrow \{1 ,2 , . . ,n\} .$  This map is well-defined because the decomposition of $(0)$ as the intersection of the $Q_{j}$'s is reduced - if two different $Q_{j}$ were equal to the same $m_{x} ,$ the decomposition would be redundant.  On the other hand, $\omega $ must be injective:  if $\omega (x) =\omega (y) =k ,$ then that means $m_{x} =Q_{k} =m_{y} ,$ whence $x =y$ by the last statement of proposition 10.  Thus we have an injective map from $X$ into a finite set $\{1 ,2 , . . . ,n\} ,$ whence $X$ is finite.  This completes the proof. \bigskip

A few remarks are in order.  First, although we are mostly interested in power set rings, many of these results hold in a general Boolean ring.  For example, in a Boolean ring, given an element $b ,$ define its \emph{complement} $b^{ \prime }$ to be $1 -b .$  Furthermore, it is always possible to define a partial order on a Boolean ring $R :$ define $a \preceq b$ to mean $ab =a .$ It is a basic exercise to show that this is a partial order and in the case when $R =\mathcal{P}(X)$, this partial order is just the containment of sets.  Relative to a partial order, the \emph{atoms} are defined to be the minimal nonzero elements.  So in $R =\mathcal{P}(X) ,$ where the partial order is given by containment, the atoms are the singleton sets.  Then the following generalization of our Proposition 10 holds: In a Boolean ring, the maximal principal ideals are precisely those generated by the complement of an atom.  (See [M], Theorem 5.21). \bigskip

Similarly, a weak version of our main result holds for a general Boolean ring.  First note that the observation made above that $1$ is the only unit in $R =\mathcal{P}(X)$ actually holds in any Boolean ring.  Indeed, if $a$ is a unit, there is a $b$ (also a unit) such that $ab =1.$  Multiply both sides by $a$ to get $a^{2}b =a ,$ but $a^{2}b =ab ,$ whence we obtain $ab =a .$  Since $ab =1 ,$ this gives $a =1 ,$ which in turn implies $b =1$ also.  

Thus, it follows that a principal ideal $(a)$ is proper if and only if $a$ is not a unit; that is, if and only if $a \neq 1$ which means $a^{ \prime } \neq 0.$  From this we can deduce that in any Boolean ring, $(0)$ is the intersection $H$ of all the maximal ideals.  Indeed, there is at least one maximal ideal (containing $(0)$) by lemma 3, so $H$ is nonempty.  So it is certainly true that $(0) \subseteq H .$ Conversely, suppose $x \neq 0.$  Then by the above remarks, $ <x^{ \prime } >$ is a proper ideal, whence it is contained in a maximal ideal $K ,$ again by lemma 3.  Thus, $x \notin K$ (otherwise both $x^{ \prime } =1 -x$ and $x$ belong to $K ,$ which would force $K =R ,$ a contradiction.)  In particular, since $x \notin K ,$ we also have $x \notin H ,$ which shows the opposite containment.  It follows that $(0) =H ,$ as claimed. \bigskip

The reason we call this a \textit{weak} version of our main result is because for a general Boolean ring, we have no idea if this is a finite intersection or an infinite one.  Of course, the homomorphic image of a Boolean ring is again a Boolean ring.  So if $I$ is any ideal in a Boolean ring and we apply the above to the ring $R/I ,$ we obtain immediately: \bigskip

\begin{corollary}
In any Boolean ring $R ,$ any ideal $I$ is equal to the intersection of all the maximal ideals which contain $I .$
\end{corollary}

This is also a known result - see for example exercise 5.34 in [M]. \bigskip

Can we extend the primary decomposition of $(0)$ to other ideals in $\mathcal{P}(X)$?  We can if the ideal is principal and is generated by a \emph{cofinite} set (that is, a set whose complement is finite.)

\begin{exercise}
Show that if $A \subseteq X ,$ then $\mathcal{P}(X)/\mathcal{P}(A)$$ \approxeq \mathcal{P}(X -A)$.  \textbf{[Hint}: Define a map $\mathcal{P}(X) \rightarrow \mathcal{P}(X -A)$ by simply dropping the elements of $A$ from the sets.  Show this is a surjective ring homomorphism with kernel $\mathcal{P}(A)$.] \bigskip 
\end{exercise}

Then we have:

\begin{corollary}
In the power set ring $\mathcal{P}(X) ,$ any principal ideal $I =(A)$ is the intersection of the maximal principal ideals which contain $I .$  If $A$ is cofinite, this gives a primary decomposition of $I$ using only principal maximal ideals.
\end{corollary}

\begin{proof}
We know every principal ideal generated by $A$ has the form $\mathcal{P}(A)$. We can refine the result of Corollary 12 where instead of using all maximal ideals containing $(0)$ in the quotient ring, we just use the principal ones, because by exercise 13, the quotient is not just a Boolean ring but also a power set ring.  Thus, $(0)$ is the intersection of all the maximal principal ideals in $\mathcal{P}(X -A)$ by our proof of the main result.  Then this pulls back in $\mathcal{P}(X)$ to writing $I =(A) =\mathcal{P}(A)$ as the intersection of the $m_{x}$ which contain $I .$  If $X -A$ is finite, this intersection is finite and hence gives a primary decomposition of $I$. \bigskip 
\end{proof}

\begin{remark}
We mentioned in the remarks following Theorem 7 that in any ring in which every prime ideal is maximal, the primary ideals which appear as factors in a primary decomposition must be unique.  Hence, part (4) of Proposition 8 implies that the primary decomposition of $I$ given in Corollary 14 is in fact the unique primary decomposition of $I$.
\end{remark}

{\textbf{5. Conclusion}. }
Looking back at our discussion in the introduction, it appears that we have succeeded in our goal to find a simpler counterexample to the Lasker-Noether theorem in non-Noetherian rings than the one that is usually cited.  By using the power set ring instead of the ring of continuous functions on a compact Hausdorff space, we seem to have avoided using any topology.  (The proof that $x \neq y$ implies $m_{x} \neq m_{y}$ was trivial in our example rather than dependent on Urysohn's lemma.) Furthermore, since our ring was Boolean, we didn't need to worry about distinguishing between primary ideals, prime ideals and maximal ideals - they were all one and the same.  So perhaps this counterexample is about as simple as one could get. \bigskip

Nevertheless, there is a sense in which our counterexample was inspired by the usual one, and indeed is almost a version of the same example.  For concreteness, if the compact Hausdorff space was some closed interval $I$ of the real line, you would not find your counterexample in the polynomial ring $\mathbb{R}[x]$(regarded as functions on $I)$, since by Hilbert's basis theorem, this ring is Noetherian!  What seems to be needed is a ring with more functions than just the polynomials, so there is a chance that the ring in question can be non-Noetherian.  Hence the appeal the the ring of all continuous functions $I \rightarrow \mathbb{R} ,$ which is quite a bit larger than $\mathbb{R}[x] .$  From this perspective, one can understand that to generalize this as much as possible, one way to go is to replace $I$ with something more general - the compact Hausdorff space $X$ in the counterexample in [AM]. \bigskip

But there is another way in which we can generalize this.  If the intention is to have a ring of functions which is `large enough to be non-Noetherian', why restrict to continuous functions at all?  Wouldn't the ring of \textit{all} functions $X \rightarrow \mathbb{R}$ suit just as well, since it is even larger than the ring of continuous functions?  (And we wouldn't have to worry about any topology since the functions need not be continuous.)  At the same time, the definition (see part (i) of the counterexample in the introduction above) of the $m_{x}$ relies only on knowing if a particular function is $0$ or not at a point $x$ - this is a binary choice, so maybe we don't care that the functions be real-valued.  Perhaps we can get away with $\mathbb{Z}_{2}$ -valued functions. \bigskip

This line of reasoning leads one to try the following example:  Let $X$ be any set, and let $S =Fun(X ,\mathbb{Z}_{2}) ,$ the ring of all functions $X \rightarrow \mathbb{Z}_{2} .$  If $X$ is infinite, we'd expect to get the same kind of counterexample as in the compact Hausdorff case for continuous functions.\bigskip

The astute reader has probably already seen the conclusion to where we are heading.  It is well known and an elementary exercise that $R =\mathcal{P}(X)$ is isomorphic to $S =Fun(X ,\mathbb{Z}_{2})$, under the isomorphism given by $A \mapsto \chi _{A} ,$ pairing a set with its characteristic function, and that under this pairing, it is easy to check that the $m_{x}$ as defined in Proposition 10 is carried over to the $m_{x}$ given in part (i) of the counterexample in the introduction above.  So we are led naturally to considering power set rings, and more generally, Boolean rings.\bigskip

\begin{center}{\Large References}\bigskip \end{center}\par
[AM]\textit{ Introduction to Commutative Algebra}, M.F. Atiyah and I.G. MacDonald, Addison-Wesley, 1969 \bigskip

[G] \textit{Contemporary Abstract Algebra}, 8th Edition, Joseph A. Gallian, Brooks-Cole, 2013 \bigskip

[H] \textit{Algebra}, T. Hungerford, Springer, 1974 \bigskip

[M] \textit{Schaum's Outline - Theory and Problems of Boolean Algebras and Switching Circuits}, E. Mendelson, McGraw-Hill, 1970
\bigskip

[N] \textit{Ideal Theory}, D.G. Northcott, Cambridge University Press, 1965 \bigskip  \bigskip

[ZS] \textit{Commutative Algebra, Vol. I}, O.Zariski and P. Samuel, van Nostrand, 1958.

\end{document}